\documentclass{amsproc}
\usepackage{amssymb}
\usepackage{amsfonts}
\usepackage{xcolor}
\usepackage{cite}
\usepackage{graphicx}

\theoremstyle{plain}

\newtheorem{theorem}{Theorem}[section]
\newtheorem{definition}[theorem]{Definition}
\newtheorem{notation}[theorem]{Notation}
\newtheorem{proposition}[theorem]{Proposition}
\newtheorem{corollary}[theorem]{Corollary}
\newtheorem{lemma}[theorem]{Lemma}

\newtheorem{remark}[theorem]{Remark}

\numberwithin{equation}{section}
\newtheorem*{theorem*}{Theorem}

\input{tcilatex}

\begin{document}
\title[]{Almost sure Assouad-like Dimensions of Complementary sets}
\author{Ignacio Garc\'ia, Kathryn Hare and Franklin Mendivil}
\address{Centro Marplatense de Investigaciones Matem\'{a}ticas, Facultad de
Ciencias Exactas y Naturales\\
and Instituto de Investigaciones F\'{\i}sicas de Mar del Plata (CONICET)\\
Universidad Nacional de Mar del Plata, Argentina }
\email{nacholma@gmail.com}
\address{Dept. of Pure Mathematics, University of Waterloo, Waterloo, Ont.,
Canada, N2L 3G1}
\email{kehare@uwaterloo.ca}
\address{Department of Mathematics and Statistics, Acadia University,
Wolfville, N.S. Canada, B4P 2R6}
\email{franklin.mendivil@acadiau.ca}
\thanks{The research of K. Hare is partially supported by NSERC Discovery
grant 2016:03719. The research of F. Mendivil is partially supported by
NSERC Discovery grant 2012:238549. I. Garc\'{\i}a and K. Hare thank Acadia
University for their hospitality when some of this research was done. I. Garc%
\'{\i}a thanks the hospitality of University of Waterloo.}

\begin{abstract}
{Given a non-negative, decreasing sequence $a$ with sum $1$, we consider all
the closed subsets of $[0,1]$ such that the lengths of their complementary
open intervals are given by the terms of $a$, the so-called }complementary
sets{. In this paper we determine the almost sure value of the }$\Phi $-{%
dimensions of these sets given a natural model of randomness. The }$\Phi $%
-dimensions are intermediate Assouad-like dimensions which include the
Assouad and quasi-Assouad dimensions as special cases. The answers depend on
the size of $\Phi ,$ with one size behaving like the Assouad dimension and
the other, like the quasi-Assouad dimension.
\end{abstract}

\maketitle

\section{Introduction}

{The upper and lower Assouad dimensions were introduced by Assouad in \cite%
{A1,A2} and Larman in \cite{L1}. These dimensions were initially used in the
theory of embeddings of metric spaces into ${\mathbb{R}}^{n}$ (see \cite{MT}%
) but, together with their less extreme versions, the quasi-Assouad
dimensions introduced in \cite{CDW, LX}, have recently been extensively used
within the fractal geometry community; see, for example, \cite%
{FTrans,FTale,FYAdv, GH,KR,Luu} and the references cited in those papers. In 
\cite{GHMPhi}, the authors further generalized these notions, introducing a
range of dimensions that are intermediate between the box and Assouad
dimensions. One topic explored in \cite{GHMPhi} were the dimensional
properties of (deterministic) rearrangements of a Cantor set of zero
Lebesgue measure. In this paper, we continue this investigation, studying
the almost sure dimensional properties of random rearrangements. This
extends the work of Hawkes \cite{Haw} who studied the Hausdorff dimensions
of random rearrangements of Cantor sets.}

{The Assouad dimensions can roughly be thought of as refinements of the
box-counting dimensions where one \textquotedblleft
localizes\textquotedblright\ and takes the worst local behaviour. The
localization is accomplished by choosing a window size, $R$, and then
analyzing this window at a smaller scale $r$. The quasi-Assouad dimension
requires, in addition, that $r\leq R^{1+\delta }$ for some fixed $\delta >1$
and then lets $\delta \rightarrow 0$. Our refinement uses a dimension
function $\Phi $ and requires $r\leq R^{1+\Phi (R)}$, so we can very
precisely measure the local scaling behaviour of a set by varying $\Phi $ to
be adapted to the set in question. The extreme values of the }$\Phi $%
-dimensions are the box and Assouad dimensions. Both the quasi-Assouad and
Assouad dimensions are special examples and when $\Phi (R)\rightarrow 0$ as $%
R\rightarrow 0,$ the $\Phi $-dimensions lie between these two. An example is
constructed in \cite{GHMPhi} of a set where the range of $\Phi $-dimensions
is the full interval from the quasi-Assouad to Assouad dimensions. Fraser's
modified $\theta $-spectrum, studied in \cite{FTale,FYAdv}, is another
example of a $\Phi $-dimension. More generally, if $\Phi $ stays bounded
away from $0,$ then the $\Phi $-dimensions lie between the box and
quasi-Assouad dimensions. The definitions and basic properties of these
dimensions are given in Section \ref{sec:Notation}.

{Given $a=\{a_{j}\}$, a non-negative decreasing sequence with sum equal to
one, we define the class ${\mathcal{C}}_{a}$ to be the family of all closed
subsets of $[0,1]$ whose complement in $[0,1]$ consists of disjoint open
intervals with lengths given by the $a_{j}$. The sets in ${\mathcal{C}}_{a}$
are called the }complementary sets{\ of $a$ and all have zero Lebesgue
measure. Every compact subset of $[0,1]$ of Lebesgue measure zero belongs to
exactly one ${\mathcal{C}}_{a}$ and each ${\mathcal{C}}_{a}$ contains both
countable and uncountable sets. Thus it is natural to ask about the possible
dimensions in a given family. }

{Besicovitch and Taylor \cite{BT} were the first to study this problem for
the case of Hausdorff dimension. Among other things, they proved that the
set of attained Hausdorff dimensions for elements of ${\mathcal{C}}_{a}$ is
the closed interval $[0,\dim _{H}C_{a}]$, where $C_{a}$ is the Cantor set in 
$\mathcal{C}_{a}$. Recent work produced similar results for the packing
dimension where the set of attainable dimensions is $[0,\dim _{P}C_{a}],$
(see \cite{HMZ}) and the upper and lower }$\Phi $-{dimensions, where under
natural technical assumptions on the sequence }$a$, {the sets of attainable
dimensions are the intervals $[\overline{\dim }_{\Phi }C_{a},1]$ and $[0,%
\underline{\dim }_{\Phi }C_{a}]$ respectively, (see \cite{GHM} for the
Assouad dimensions and \cite{GHMPhi} for the more general }$\Phi $-dimensions%
{). }

{An alternative thread, started in \cite{Haw}, found the almost sure
Hausdorff dimension for a random element of ${\mathcal{C}}_{a},$ under a
very natural model of randomness. This was extended in \cite{Hu1,Hu2} where
the exact almost sure Hausdorff and packing dimension functions were found
for the same random model. Note that since the value of the dimension of a
set depends on the asymptotics of very fine scales, any dimensional
calculation will be a tail event and thus will have a constant value almost
surely (at least if the randomness is given by appropriately independent
choices). In this paper, we determine the upper and lower }$\Phi $%
-dimensions of these random rearrangements. Surprisingly, {the almost sure
behaviour of the dimension depends on the asymptotic \textquotedblleft
size\textquotedblright\ of $\Phi $. }In fact, it is this difference in the
almost sure behaviour which motivated us to study the $\Phi $-dimensions.

{Our main results, which can be found in Sections \ref{sec:asupperdim} and %
\ref{sec:aslowerdim}}, can be summarized as follows:

%{\renewcommand{\thetheorem}{\ref{asupperlarge}} }

\begin{theorem*}
Let $a$ be a level comparable sequence and let $\Psi (x)=\log |\log x|/|\log
x|$.

(i) If $\Phi (x)>>\Psi (x)$ for $x$ near $0,$ then for a.e. $E\in \mathcal{C}%
_{a}$ we have 
\begin{equation*}
\overline{\dim }_{\Phi }E=\overline{\dim }_{\Phi }C_{a}\ \ \text{ and }\ \ 
\underline{\dim }_{\Phi }E=\underline{\dim }_{\Phi }C_{a}.
\end{equation*}

(ii) If $\Phi (x)<<\Psi (x)$ for $x$ near $0,$ then for a.e. $E\in \mathcal{C%
}_{a}$ we have 
\begin{equation*}
\overline{\dim }_{\Phi }E=1\ \ \text{ and}\ \ \ \underline{\dim }_{\Phi }E=0.
\end{equation*}
\end{theorem*}

We do not know what happens if $\Phi \sim \Psi $. We note that the values $%
1,0,$ that arise in the small case, are the upper and lower $\Phi $%
-dimensions of the countable decreasing set that belongs to $\mathcal{C}_{a}$%
.

Since the Assouad dimensions are examples of \textquotedblleft
small\textquotedblright\ $\Phi $-dimensions, while the quasi-Assouad
dimensions are examples of \textquotedblleft large\textquotedblright\ $\Phi $%
-dimensions, we immediately deduce:

\begin{corollary}
Let $a$ be a level comparable sequence. Then for a.e. $E\in \mathcal{C}_{a}$
we have 
\begin{equation*}
\dim _{qA}E=\dim _{qA}C_{a},\dim _{qL}E=\dim _{qL}C_{a}
\end{equation*}%
and 
\begin{equation*}
\dim _{A}E=1,\dim _{L}E=0.
\end{equation*}
\end{corollary}

The proofs of these results are very different from both the deterministic
arguments and the earlier random results, and rely heavily upon
probabilistic information about the tails of binomial distributions. A very
loose interpretation of these results is that the quasi-Assouad dimensions
require consideration of deep enough scales for the Central limit theorem to
\textquotedblleft reveal\textquotedblright\ itself so that the almost sure
dimension coincides with the dimension of $C_{a},$ the \textquotedblleft
average\textquotedblright\ set.

\section{Background\label{sec:Notation}}

\subsection{Definition and examples of $\Phi $-dimensions}

Given a metric space $X,$ we denote the ball centred at $x\in X$ with radius 
$R$ by $B(x,R)$. For a bounded set $E\subseteq X$, the notation $N_{r}(E)$
will mean the least number of balls of radius $r$ that cover $E$.

\begin{definition}
By a \textbf{dimension function}, we mean a map $\Phi :(0,1)\mathbb{%
\rightarrow R}^{+}$ with the property that $R^{1+\Phi (R)}$ decreases as $R$
decreases to $0$.
\end{definition}

Of course, $R^{1+\Phi (R)}\leq R$, so $R^{1+\Phi (R)}\rightarrow 0$ as $%
R\rightarrow 0$ for any dimension function $\Phi $. As will be seen,
interesting examples of dimension functions include the constant functions,
as well as $\Phi (x)=1/|\log x|$ and $\Phi (x)=\log |\log x|/|\log x|$.

\begin{definition}
Let $\Phi $ be a dimension function and $X$ a metric space. The \textbf{%
upper }and \textbf{lower }$\Phi $-\textbf{dimensions\ }of $E\subseteq X$ are
given by 
\begin{align*}
\overline{\dim }_{\Phi }E=\inf \Bigl\{\alpha :(\exists c_{1},c_{2}>0)&
(\forall 0<r\leq R^{1+\Phi (R)}\leq R<c_{1})\text{ } \\
& \sup_{x\in E}N_{r}(B(x,R)\tbigcap E)\leq c_{2}\left( \frac{R}{r}\right)
^{\alpha }\Bigr\}
\end{align*}%
and 
\begin{align*}
\underline{\dim }_{\Phi }E=\sup \Bigl\{\alpha :(\exists c_{1},c_{2}>0)&
(\forall 0<r\leq R^{1+\Phi (R)}\leq R<c_{1})\text{ } \\
& \sup_{x\in E}N_{r}(B(x,R)\tbigcap E)\geq c_{2}\left( \frac{R}{r}\right)
^{\alpha }\Bigr\}.
\end{align*}
\end{definition}

The $\Phi $-dimensions were first introduced in \cite{GHMPhi} where their
basic properties were established. Some of these will be highlighted below.

\medskip

Special examples of $\Phi $-dimensions include the following:

(i) The \textbf{upper} \textbf{Assouad }and \textbf{lower Assouad dimensions 
}of $E$, denoted $\dim _{A}E$ and $\dim _{L}E$ respectively. These are the
special cases of the upper and lower $\Phi $-dimensions with $\Phi =0$.

(ii) The (modified) \textbf{upper }and \textbf{lower }$\theta $-\textbf{%
spectrum}, $\overline{\dim }_{A}^{\theta }E$ and $\underline{\dim }%
_{L}^{\theta }E,$ introduced by Fraser in \cite{FYAdv}, arise by taking the
constant function $\Phi =1/\theta -1$. More generally, it is shown in \cite%
{GHMPhi} that if $\Phi (x)\rightarrow 1/\theta -1$ as $x\rightarrow 0$, then
the $\Phi $-dimensions coincide with the $\theta $-spectrum.

(iii) The \textbf{upper} \textbf{quasi-Assouad} and \textbf{lower
quasi-Assouad dimensions}, denoted $\dim _{qA}E$ and $\dim _{qL}E$,\textbf{\ 
}are defined as the limit as $\delta \rightarrow 0$ of the upper and lower $%
\Phi _{\delta }=\delta $ dimensions, respectively. For every set $E$ there
are dimension functions $\Phi _{1},\Phi _{2}$ such that $\overline{\dim }%
_{\Phi _{1}}E=\dim _{qA}E$ and $\underline{\dim }_{\Phi _{2}}E=\dim _{qL}E$,
see \cite[Proposition 2.11]{GHMPhi}. But the choice of dimension functions
depends on the set $E$.

\begin{remark}
Since a set and its closure have the same $\Phi $-dimensions, unless we say
otherwise we will assume all sets are compact. We will also assume the
underlying metric space $X$ is doubling. This ensures, in particular, that $%
\dim _{A}X$ is finite.
\end{remark}

\subsection{Basic properties of $\Phi $-dimensions}

The following relationships between these dimensions are known (see \cite%
{Fal,FTrans,GHMPhi,LX}):%
\begin{equation*}
\dim _{L}E\leq \dim _{qL}E\leq \dim _{H}E\leq \underline{\dim }_{B}E\leq 
\overline{\dim }_{B}E\leq \dim _{qA}E\leq \dim _{A}E
\end{equation*}%
and 
\begin{equation*}
\dim _{L}E\leq \underline{\dim }_{\Phi }E\leq \underline{\dim }_{B}E\leq 
\overline{\dim }_{B}E\leq \overline{\dim }_{\Phi }E\leq \dim _{A}E.
\end{equation*}

\bigskip Here are some other facts which were shown in \cite[Section 2]%
{GHMPhi}.

\begin{proposition}
(i) If $\Phi (x)\rightarrow \infty $ as $x\rightarrow 0,$ then $\overline{%
\dim }_{B}E=\overline{\dim }_{\Phi }E$. If, in addition, $\underline{\dim }%
_{\Phi }E>0$, then $\underline{\dim }_{\Phi }E=\underline{\dim }_{B}E$.
Without the additional assumption, the latter statement need not be true
since any set with an isolated point will have $\underline{\dim }_{\Phi }E=0$%
.

(ii) If $\Phi \leq \Psi ,$ then $\underline{\dim }_{\Phi }E\leq \underline{%
\dim }_{\Psi }E$ and $\overline{\dim }_{\Phi }E\geq \overline{\dim }_{\Psi }$%
. In particular, if $\Phi (R)\rightarrow 0$ as $R\rightarrow 0$, then the $%
\Phi $-dimensions give a range of dimensions between the Assouad and
quasi-Assouad type dimensions: 
\begin{equation*}
\text{ }\dim _{L}E\leq \underline{\dim }_{\Phi }E\leq \dim _{qL}E\leq \dim
_{qA}E\leq \overline{\dim }_{\Phi }E\leq \dim _{A}E\text{.}
\end{equation*}

(iii) If $\Phi (x)\leq c/|\log x|$ for all small $x$, then the $\Phi $%
-dimensions coincide with the Assouad dimensions.
\end{proposition}

Many examples have been constructed to illustrate strict inequalities
between these dimensions. For instance, although the $\Phi $ and $\Psi $
dimensions coincide for all sets $E$ if $\Phi /\Psi \rightarrow 1$ as $%
x\rightarrow 0$, there will be sets where these dimensions differ if $\Phi $
is bounded above away from $\Psi $. Moreover, given $0<\alpha <\beta <1$,
there is a set $E\subseteq \mathbb{R}$ such that%
\begin{equation*}
\{\overline{\dim }_{\Phi }E:\Phi \rightarrow 0\}=[\alpha ,\beta ]=[\dim
_{qA}E,\dim _{A}E];
\end{equation*}%
\cite[Theorems 3.6, 3.7]{GHMPhi}.

It is easy to see that the $\Phi $-dimensions are bi-Lipschitz invariant and
give detailed geometric information about the structure of the underlying
sets.

The following notation will be convenient for later in the paper.

\begin{notation}
We write $f\sim g,$ and say $f$ is comparable to $g,$ if
there are positive constants $c_{1},c_{2}$ such that $c_{1}f\leq g\leq c_{2}f
$. The symbols $\gtrsim $ and $\lesssim $ are defined similarly. When we
write $f<<g,$ this means $f/g\rightarrow 0$ as either $x\rightarrow 0$ or $%
n\rightarrow \infty ,$ depending on the context.
\end{notation}

\section{Complimentary sets and the Random model}

\label{sec:complementarysets}

\label{subsec:complsets}

\subsection{Complementary sets and the associated Cantor set}

The focus of this paper will be on the relationship between the dimensions
of compact subsets of $\mathbb{R}$ whose complements are open intervals of
the same length. We refer to these as complementary sets or rearrangements
and begin by explaining precisely what we mean by that.

Every closed subset of the interval $[0,1]$ of Lebesgue measure zero is of
the form $E=[0,1]\diagdown \tbigcup U_{j}$ where $\{U_{j}\}$ is a disjoint
family of open subintervals of $[0,1]$ whose lengths sum to one. Let $a_{j}$
be the length of $U_{j}$. There is no loss of generality in assuming $%
a=(a_{j})$ is a decreasing sequence. We denote by $\mathcal{C}_{a}$ the
collection of all such closed sets $E$; the sets in $\mathcal{C}_{a}$ are
called the \textbf{complementary sets of }$a$. Every family, $\mathcal{C}%
_{a},$ contains a countable set, the decreasing rearrangement, $%
D_{a}=\{\sum_{i\geq k}a_{i}\}_{k=1}^{\infty }$.

\medskip

Another complementary set in $\mathcal{C}_{a}$ is the so-called \textbf{%
Cantor set associated with }$a$ and denoted by $C_{a}$. It is constructed as
follows: In the first step, we remove from $[0,1]$ an open interval of
length $a_{1}$, resulting in two closed intervals $I_{1}^{1}$ and $I_{2}^{1}$%
. Having constructed the $k$-th step, we obtain the closed intervals $%
I_{1}^{k},...,I_{2^{k}}^{k}$ contained in $[0,1]$. The intervals $I_{j}^{k},$
$j=1,...,2^{k},$ are called the Cantor intervals of step $k$. The next step
consists in removing from each $I_{j}^{k}$ an open interval of length $%
a_{2^{k}+j-1}$, obtaining the closed intervals $I_{2j-1}^{k+1}$ and $%
I_{2j}^{k+1}$. We define 
\begin{equation*}
C_{a}:=\bigcap_{k\geq 1}\bigcup_{j=1}^{2^{k}}I_{j}^{k}.
\end{equation*}

This construction uniquely determines the set because the lengths of the
removed intervals on each side of a given gap are known. For instance, the
classical middle-third Cantor set is the Cantor set associated with the
sequence $a=\{a_{i}\}$ where $a_{i}=3^{-n}$ if $2^{n-1}\leq i\leq 2^{n}-1$.
This sequence $a$ is \textbf{doubling}, meaning there is a constant $\kappa $
such that $a_{n}\leq \kappa a_{2n}$ for all $n$. Whenever $a$ is doubling,
then $C_{a}$ is bi-Lipschitz equivalent to the central Cantor set $C_{b}$
where $b_{2^{n}}=a_{2^{n}}$ and central means all interval (equivalently,
gaps) on the same level have the same length.

\subsection{Dimensional properties of complementary sets}

All complementary sets have the same box dimensions (see \cite{Fal}),
however, this is not true for the other dimensions. For instance, $\dim
_{H}D_{a}=0,$ but this is not true in general for $C_{a}$. Thus it is of
interest to study the dimensional properties of complementary sets. This
investigation began with Besicovitch and Taylor in \cite{BT} where they
proved that the Cantor set associated with $a$ has the maximal Hausdorff
dimension of all sets in $\mathcal{C}_{a}$. Moreover, they showed that given
any $s\in \lbrack 0,\dim _{H}C_{a}]$, there was some $E\in \mathcal{C}_{a}$
with $\dim _{H}E=s$. The analogous result was subsequently shown in \cite%
{HMZ} for packing dimension.

In \cite{GHM}, this problem was studied for the Assouad dimensions with the
same result again true for the lower Assouad dimension. However, for the
upper Assouad dimension, it was discovered that the associated Cantor set
had the minimal Assouad dimension of all complementary sets and the
decreasing set had the maximal dimension. Under the assumption that $a$ is
doubling, it was shown that set of attainable values for the upper Assouad
dimension was the full interval $[\dim _{A}C_{a},\dim _{A}D_{a}]=[\dim
_{A}C_{a},1]$. Under a slightly stronger assumption, implied by level
comparable (defined below) the set of attainable lower Assouad dimensions
was also shown to be the interval $[0,\dim _{L}C_{a}]$.

\begin{definition}
Given a decreasing sequence $a$, let $s_{n}=2^{-n}\sum_{j\geq 2^{n}}a_{j},$
the average length of the Cantor intervals of $C_{a}$ of step $n$. We will
say the doubling sequence $a$ is \textbf{level comparable} if there are
constants $\tau $ and $\lambda $ with%
\begin{equation}
0<\tau \leq s_{j+1}/s_{j}\leq \lambda <1/2.  \label{doubling}
\end{equation}
\end{definition}

For a central Cantor set, level comparable simply means the ratios of
dissection, $s_{j+1}/s_{j},$ are bounded away from $0$ and $1/2$. We should
point out that the doubling condition already ensures the left hand
inequality holds in (\ref{doubling}). The level comparable condition is very
helpful as it implies that $s_{k}\sim a_{2^{k}}$ because $s_{k}\geq
a_{2^{k+1}}\gtrsim a_{2^{k}}$ and 
\begin{equation}
(1-2\lambda )s_{k}\leq s_{k}-2s_{k+1}\leq a_{2^{k}}.  \label{LevelCompInfo}
\end{equation}

In \cite{GHMPhi}, the $\Phi $-dimensions of rearrangements were
investigated. The results are similar to the Assouad dimensions (although
new proofs were needed in some cases).

\begin{theorem}
\cite[Cor. 4.2, Theorem 4.3, Cor. 4.4]{GHMPhi}\label{Recall} If $\Phi $ is
any dimension function and $a$ a decreasing, summable sequence, then $%
\overline{\dim }_{\Phi }E\leq \overline{\dim }_{\Phi }D_{a}$. If $a$ is
level comparable, then $\overline{\dim }_{\Phi }E\geq \overline{\dim }_{\Phi
}C_{a}$ and \underline{$\dim $}$_{\Phi }C_{a}\geq \underline{\dim }_{\Phi }E$%
. If, in addition, $\Phi \rightarrow p\in \lbrack 0,\infty ],$ then 
\begin{eqnarray*}
\{\underline{\dim }_{\Phi }E &:&E\in \mathcal{C}_{a}\}=[0,\underline{\dim }%
_{\Phi }C_{a}] \\
\{\overline{\dim }_{\Phi }E &:&E\in \mathcal{C}_{a}\}=[\overline{\dim }%
_{\Phi }C_{a},\overline{\dim }_{\Phi }D_{a}].
\end{eqnarray*}
\end{theorem}

One ingredient in the proof was a formula for computing the $\Phi $%
-dimensions of Cantor sets. For this, it is helpful to understand the
comparison $r\leq R^{1+\Phi (R)}$ in terms of the sequence $(s_{n})$.

\begin{notation}
Given a dimension function $\Phi (x)$ and a doubling sequence $a=(a_{j}),$
we define the \textbf{depth function} $\phi $ on $\mathbb{N}$ by the rule
that $\phi (n)$ is the minimal integer $j$ such that $s_{n+j}\leq
s_{n}^{1+\Phi (s_{n})}$. In other words, $\phi (n)$ is the minimal integer
with $s_{n+\phi (n)}/s_{n}\leq s_{n}^{\Phi (s_{n})}$.
\end{notation}

One can easily check that if $\phi (n)\geq 2,$ then $\phi (n)\sim n\Phi
(s_{n})$. Thus if $\phi (n)/n\rightarrow \infty $, then $\overline{\dim }%
_{\Phi }E=\overline{\dim }_{B}E,$ while if $\phi $ is bounded, then the
upper (or lower) $\Phi $-dimension coincides with the upper (resp., lower)
Assouad dimension.

\begin{theorem}
\label{Cantorformula}\cite[Theorem 3.3]{GHMPhi} Let $a$ be a doubling
sequence and $C_{a}$ the associated Cantor set. The upper and lower $\Phi $%
-dimensions of $C_{a}$ are given by 
\begin{equation}
\overline{\dim }_{\Phi }C_{a}=\inf \Bigl\{\beta :(\exists \ k_{0},c_{0}>0)\
(\forall k\geq k_{0}\text{, }n\geq \phi (k))\text{ }\left( \frac{s_{k}}{%
s_{k+n}}\right) ^{\beta }\geq c_{0}2^{n}\Bigr\}  \label{Thetadim}
\end{equation}%
and%
\begin{equation}
\underline{\dim }_{\Phi }C_{a}=\sup \Bigl\{\beta :(\exists \ k_{0},c_{0}>0)\
(\forall k\geq k_{0}\text{, }n\geq \phi (k))\text{ }\left( \frac{s_{k}}{%
s_{k+n}}\right) ^{\beta }\leq c_{0}2^{n}\Bigr\}.  \label{LowerTheta}
\end{equation}
\end{theorem}

\subsection{Random Model for Complementary sets}

\label{Sec:asresults}

The goal of this paper is to study the almost sure dimensional properties of
random rearrangements. We now describe the model that we use to generate a
random ordering of $\mathbb{N}$ and thereby a random set belonging to the
sequence $\{a_{n}\}$. Our approach is formally different from that of Hawkes
in \cite{Haw}, but the resulting random ordering is the same, as we explain
below.

The random order has two salient and defining features: 1) when it is
restricted to any finite subset of $\mathbb{N}$ each possible ordering is
equally likely, and 2) for any two disjoint $A,B\subset \mathbb{N}$, the
random order restricted to $A$ is independent of the one restricted to $B$.

Our construction is inductive. We start the induction with the trivial order
on the set $\{ 1 \}$. Having constructed a random order on $\{1,2,\ldots,
2^n-1\}$, the induction step consists of two parts which are done
independently:

\begin{enumerate}
\item Choose a (uniformly) random permutation of $\{2^n, 2^n+1, \ldots,
2^{n+1}-1\}$, and

\item Randomly choose, independently and with replacement, a set of $2^{n}$
\textquotedblleft locations\textquotedblright\ in which to insert the
elements of $\{2^{n},2^{n}+1,\ldots ,2^{n+1}-1\}$.
\end{enumerate}

The idea is that extending a permutation of $\{1,2,\ldots ,2^{n}-1\}$ to a
permutation of $\{1,2,\ldots ,2^{n}-1,\ldots ,2^{n+1}-1\}$ involves ordering 
$\{2^{n},2^{n}+1,\ldots ,2^{n+1}-1\}$ and then inserting these elements into
the already existing permutation (including left of the left-most or right
of the right-most). Each of these \textquotedblleft
places\textquotedblright\ could contain zero or more \textquotedblleft
new\textquotedblright\ elements. This means that 2) above is equivalent to
generating a sample from the multinomial distribution of $2^{n}$ trials with 
$2^{n}$ outcomes (the \textquotedblleft locations\textquotedblright ) which
are all equally likely.

Let $\mathcal{O}$ be the set of all total orders on $\mathbb{N}$ and for a
finite subset $F\subset \mathbb{N}$ and a total order $\triangleleft $ on $F$%
, let $\mathcal{O}_{\triangleleft }=\{\prec \in \mathcal{O}\,:\,\prec
|_{F}=\triangleleft \}$ (these are the analogues in this situation of the
\textquotedblleft cylinder sets\textquotedblright\ from a countable
product). The $\sigma$-algebra we use is generated by the sets $\mathcal{O}%
_{\triangleleft }$ taken over all finite sets $F$ and over all total orders $%
\triangleleft $ on $F$. Our probability measure on $\mathcal{O}$ is
generated by the property that $\text{Prob}(\mathcal{O}_{\triangleleft
})=(|F|!)^{-1}$.

Given a total order $\triangleleft $ $\in \mathcal{O}$ and $a=\{a_{n}\},$ we
define, for each $i\in \mathbb{N}$, a random open interval of length $a_{i}$
by 
\begin{equation*}
J_{i}(\triangleleft ):=(\sum_{j\triangleleft
i}a_{j},a_{i}+\sum_{j\triangleleft i}a_{j}).
\end{equation*}%
We define the random set $K_{\triangleleft }\in \mathcal{C}_{a}$ by 
\begin{equation}
K_{\triangleleft }:=[0,\sum_{i}a_{i}]\setminus \{\tbigcup
J_{i}(\triangleleft )\}.  \label{eq:Komega}
\end{equation}%
Notice in particular that if $i\triangleleft j$ then $J_{i}(\triangleleft )$%
, the gap corresponding to $a_{i}$, is to the left of $J_{j}(\triangleleft )$%
, the gap corresponding to $a_{j}$. The subintervals of $[0,1]$ that are
bounded by the gaps of lengths $a_{1},...,a_{2^{n}-1}$ (or the unbounded
gaps) will be called the intervals of level $n$ for the set $%
K_{\triangleleft }$. We remark that almost surely a set in $\mathcal{C}_{a}$
has no isolated points and hence there are $2^{n}$ closed intervals at step $%
n$ in this construction.

The model from \cite{Haw} generates a random order on $\mathbb{N}$ by
choosing an iid sequence, $\omega _{n}$, of $U[0,1]$ random variables and
defining $i\prec _{\omega }j$ if and only if $\omega _{i}\leq \omega _{j}$.
Hawkes' random set $K_{\omega }$ is our set $K_{\prec _{\omega }}$.
Obviously, $K_{\omega }=K_{\omega ^{\prime }}$ if and only if the
corresponding total orders, $\prec _{\omega }$ and $\prec _{\omega
^{\prime}},$ agree. When restricted to a finite subset $F\subset \mathbb{N}$
each possible order is equally likely, so if we are given $\omega \in
\lbrack 0,1]^{F}$ and let $\mathcal{O}_{\omega }=\{\omega ^{\prime }:\prec
_{\omega ^{\prime }}|_{F}=\prec _{\omega }\}$, then Prob$(\mathcal{O}%
_{\omega })=(|F|!)^{-1}$. The $\sigma $-algebra on $\mathcal{O}$ is also
generated by the various $\mathcal{O}_{\omega }$, $\omega \in \lbrack
0,1]^{F},$ taken over all finite sets $F$ since the product $\sigma $%
-algebra on $[0,1]^{\infty }$ is generated by the cylinder sets. Thus the
two random models are effectively the same.

\medskip

Our proofs will rely heavily upon the following variation on the
DeMoivre-Laplace theorem \cite[p.13, Theorem 7]{Bo}.

\begin{theorem}
If $Y$ is a binomially distributed random variable with distribution $%
B(M,2^{-N})$ and $\eta 2^{-N}(1-2^{-N})M\geq 12$ for some $\eta <1/12\mathtt{%
,}$ then%
\begin{equation*}
\mathcal{P(}\left\vert Y-M2^{-N}\right\vert \geq \eta M2^{-N})\leq \exp
\left( -\eta ^{2}M2^{-N}/3\right) /(\eta \sqrt{M2^{-N}}).
\end{equation*}
\end{theorem}

Specifically, we will use the following corollary.

\begin{corollary}
\label{prob}There is a constant $c>0$ such that if $Y$ is a binomially
distributed random variable with distribution $B(M,2^{-N})$ and $M2^{-N}\geq
200$, then 
\begin{equation*}
\mathcal{P(}Y<<M2^{-N})\text{, }\mathcal{P(}Y>>M2^{-N})\leq \exp \left(
-cM2^{-N}\right) .
\end{equation*}
\end{corollary}

\begin{proof}
These follow immediately from the theorem upon noting that the set $\left\{
Y>>M2^{-N}\right\} $ is contained in $\left\{ Y\geq (13/12)M2^{-N}\right\} $
and the set $\left\{ Y<<M2^{-N}\right\} $ is contained in $\left\{ Y\leq
(11/12)M2^{-N}\right\} .$
\end{proof}

\section{Almost Sure Upper Dimensions for Complementary sets}

\label{sec:asupperdim}

\textbf{Terminology}: To study the dimensional properties of random
rearrangments, it will helpful to introduce the following iterative
construction of $E\in \mathcal{C}_{a}$ that we will refer to as the \textbf{%
standard construction}. We begin with the interval $[0,1]$ and then remove
the open interval (gap) $G_{1}\subseteq E^{c},$ of length $a_{1}$, leaving $%
E_{1}=[0,1]\diagdown G_{1}$, a union of at most two closed intervals called
the intervals of step one. At step (or level) 2 we remove the two gaps of
level two, $G_{2},G_{3},$ of lengths $a_{2}$ and $a_{3}$, leaving $%
E_{2}=E_{1}\diagdown (G_{2}\cup G_{3})$, a union of at most 4 closed
intervals. Now repeat this process. Given $E_{n-1}$, to form $E_{n}$ we
remove from $E_{n-1}$ the $2^{n-1}$ gaps of level $n$ of lengths $%
a_{2^{n-1}},...,a_{2^{n}-1}$, leaving $E_{n}$, a union of at most $2^{n}$
closed intervals, the intervals of step $n$. The set $E$ equals $\tbigcap
E_{n}$. Each set $E_{n}$ is the union of the finitely many closed intervals
and isolated points that lie between the gaps that have been removed at
levels $1,...,n$.

\subsection{Almost sure upper dimensions for \textquotedblleft
large\textquotedblright\ $\Phi $}

\label{subsec:upperlarge}

\begin{theorem}
\label{TheoremA} \label{asupperlarge}Let $a=\{a_{n}\}$ be a level comparable
sequence. \ For almost every $E\in \mathcal{C}_{a}$ we have $\overline{\dim }%
_{\Phi }E=\overline{\dim }_{\Phi }C_{a}$ if 
\begin{equation*}
\Phi (x)>>\frac{\log \left\vert \log (x)\right\vert }{|\log x|}\text{ for }x%
\text{ near }0,
\end{equation*}%
equivalently, $\phi (n)>>\log n$.
\end{theorem}

\begin{corollary}
\label{qA}For almost all rearrangements $E\in \mathcal{C}_{a}$, $\dim
_{qA}E=\dim _{qA}C_{a}$.
\end{corollary}

\begin{proof}
For each $\delta >0$ and $\Phi (x)=\delta ,$ we have $\overline{\dim }_{\Phi
}E=\overline{\dim }_{\Phi }C_{a}$ a.s. Letting $\delta \rightarrow 0$ gives
the result for the quasi-Assouad dimension.
\end{proof}

\begin{proof}[Proof of Theorem \ref{TheoremA}] As we noted in Theorem \ref{Recall},
in \cite{GHMPhi} it was shown that $\overline{\dim }_{\Phi }E\geq \overline{%
\dim }_{\Phi }C_{a}$ for all $E\in \mathcal{C}_{a}$. Thus it suffices to
show the other inequality holds almost surely. Fix $d>\overline{\dim }_{\Phi
}C_{a}$ and we will show $\overline{\dim }_{\Phi }E\leq d$ a.s.

Since the sequence $a$ is level comparable, there are constants $\tau
,\lambda $ such that $0<\tau \leq s_{j+1}/s_{j}\leq \lambda <1/2$. Consider $%
N_{r}(B(x,R)\tbigcap E)$ for 
\begin{equation*}
(1-2\lambda )s_{n+1}\leq R\leq (1-2\lambda )s_{n},\text{ }r\leq R^{1+\Phi
(R)}
\end{equation*}%
and $x\in E$. As $R\leq s_{n}$, we have $s_{n+m+1}\leq r<s_{n+m}$ for some $%
m\geq \phi (n)$.

Since $x\in E$, $x$ does not belong to any of the removed gaps and so $x$
must belong to some interval, $I_{n}(x),$ that arises at step $n$ in the
standard construction of $E$. As pointed out in (\ref{LevelCompInfo}), $%
(1-2\lambda )s_{n}\leq a_{2^{n}}$. Since the gaps that bound $I_{n}(x)$ (one
of which may be unbounded) have length at least $a_{2^{n}},$ we see that $%
B(x,R)\tbigcap E$ $\subseteq I_{n}(x)$, so 
\begin{equation*}
N_{r}(B(x,R)\tbigcap E)\leq N_{r}(I_{n}(x)\tbigcap E).
\end{equation*}%
As $R/r\sim s_{n}/s_{n+m}$ it will be enough to prove that almost surely $%
N_{r}(I_{n}(x)\tbigcap E)\leq 2\left( \frac{s_{n}}{s_{n+m}}\right) ^{d}$ for 
$n$ sufficiently large. In other words, we want to prove that it is with
probability zero that for infinitely many $n$ there are intervals $I_{n}$ of
level $n$ and integers $m\geq \phi (n)$ such that 
\begin{equation*}
N_{s_{n+m}}(I_{n}\tbigcap E)>2\left( \frac{s_{n}}{s_{n+m}}\right) ^{d}.\text{
}
\end{equation*}%
This will be a Borel Cantelli argument.

To begin, choose $L=L(n,m)$ so that 
\begin{equation*}
\sum_{j=2^{n+m+L}}^{\infty }a_{j}=2^{n+m+L}s_{n+m+L}\leq s_{n+m+1}\text{.}
\end{equation*}%
Note that $2^{i}s_{i}\leq (2\lambda )^{i}$ and $s_{n+m}\geq \tau ^{n+m}$, so
we may take $L=C(n+m)$ for a suitable constant $C>0$.

Temporarily fix interval $I_{n}$. It is known that $N_{r}(I)$ $\sim
N_{r}(I^{\prime })$ whenever $I,I^{\prime }$ are rearrangements of the same
set of gaps (see \cite[Lemma 4.1]{GHMPhi}), hence there is no loss of
generality in assuming the gaps which are placed in $I_{n}$ in the
construction of $E$ at subsequent (deeper) levels, are placed in decreasing
order.

The choice of $L$ ensures that the gaps placed in $I_{n}$ after level $%
n+m+L-1$ have total length at most $r$ and thus one interval of length $r$
will cover these in totality.

Let 
\begin{eqnarray*}
\Lambda _{0}(I_{n}) &=&\#\text{ gaps in }I_{n}\text{ from levels }n+1,...,n+m%
\text{ } \\
\Lambda _{k}(I_{n}) &=&\#\text{ gaps in }I_{n}\text{ from level }n+m+k\text{
for }k=1,...,L-1.
\end{eqnarray*}%
The gaps of level $n+m+k$ each have length comparable to $s_{n+m+k}$ and
thus for each $k$, the totality of these gaps will be covered by%
\begin{equation*}
\Lambda _{k}(I_{n})\frac{s_{n+m+k}}{s_{n+m}}
\end{equation*}%
intervals of length $r$. The gaps of levels $n+1,...,n+m$ can be covered by $%
\Lambda _{0}(I_{n})$ intervals of length $r$, thus an upper bound on $%
N_{r}(I_{n}\tbigcap E)$ is given by 
\begin{equation*}
N_{r}(I_{n}\tbigcap E)\lesssim \Lambda _{0}(I_{n})+\sum_{k=1}^{L-1}\Lambda
_{k}(I_{n})\frac{s_{n+m+k}}{s_{n+m}}+1.
\end{equation*}%
We wish to compare this to%
\begin{equation*}
\left( \frac{R}{r}\right) ^{d}\sim \left( \frac{s_{n}}{s_{n+m}}\right) ^{d}
\end{equation*}

If $N_{r}(I_{n}\tbigcap E)>2\left( \frac{s_{n}}{s_{n+m}}\right) ^{d}$, then 
\begin{equation}
\Lambda _{k}(I_{n})\frac{s_{n+m+k}}{s_{n+m}}>\frac{1}{L}\left( \frac{s_{n}}{%
s_{n+m}}\right) ^{d}  \label{Lk}
\end{equation}%
for some $k=0,1,...,L-1.$

There are $2^{n}$ such intervals $I_{n}$ for each $n$; temporarily label
them as $I_{n}^{(i)},$ $i=1,...,2^{n}$. The probability that at least one of
these intervals, $I_{n}^{(i)},$ satisfies condition (\ref{Lk}) for some $%
k=0,1,...,L,$ is at most $\sum_{i=1}^{2^{n}}p_{i,k}$ where%
\begin{equation*}
p_{i,k}=\mathcal{P}\left( \Lambda _{k}(I_{n}^{(i)})\geq \frac{1}{L}\left( 
\frac{s_{n+m}}{s_{n+m+k}}\right) \left( \frac{s_{n}}{s_{n+m}}\right) ^{d}%
\text{ }\right) .
\end{equation*}

Choose $\varepsilon >0$ such that $d-\varepsilon >\overline{\dim }_{\Phi
}C_{a}$. The choice of $\phi (n)$ and the formula for the upper $\Phi $%
-dimension of a Cantor set (\ref{Thetadim}) ensures that for large enough $n$
and all $m\geq \phi (n)$, 
\begin{equation*}
\left( \frac{s_{n}}{s_{n+m}}\right) ^{d-\varepsilon }\geq 2^{m}.
\end{equation*}%
Since $s_{i}/s_{i+1}\geq 1/\lambda =b>2,$ 
\begin{equation*}
\frac{s_{n+m}}{s_{n+m+k}}\geq b^{k}\text{ and }\left( \frac{s_{n}}{s_{n+m}}%
\right) ^{\varepsilon }\geq b^{\varepsilon m}\text{.}
\end{equation*}%
Thus for each $i=1,...,2^{n}$ and $k=0,...,L-1,$ 
\begin{equation*}
p_{i,k}\leq \mathcal{P}\left( \Lambda _{k}(I_{n}^{(i)})\geq \frac{1}{L}2^{m}%
\text{ }b^{k}b^{\varepsilon m}\right)
\end{equation*}

Let $Y_{i,k}$ be the random variable that counts the number of gaps of level 
$n+m+k$ for $k\in \{1,...,L-1\}$ (or levels $n+1,...,n+m$ if $k=0$) in
interval $I_{n}^{(i)}$ given that the gaps are placed uniformly among the $%
2^{n}$ such intervals. With this notation 
\begin{equation*}
p_{i,k}\leq \mathcal{P}\left( Y_{i,k}\geq \frac{1}{C(n+m)}2^{m}\text{ }%
b^{k}b^{\varepsilon m}\right)
\end{equation*}%
(with the appropriate modification for $k=0$.) The function $Y_{i,k}$ is a
binomially distributed random variable with distribution $%
B(2^{n+m+k-1},2^{-n})$ (or $B\left( \sum_{i=1}^{m}2^{n+i},2^{-n}\right) $ if 
$k=0$) as there are $2^{n+m+k-1}$ gaps (or $\sum_{i=1}^{m}2^{n+i-1}$ many
gaps) to be placed in $2^{n}$ positions, thus our strategy to bound $p_{i,k}$
is to use Corollary \ref{prob}, taking the $M$ in that corollary to be $%
2^{n+m+k-1}$ (with the obvious modification if $k=0$), and the $N$ to be $n, 
$ so $M2^{-N}\sim 2^{m+k}$. \texttt{\ }

Since $b>2$, incorporating this notation gives 
\begin{equation*}
\frac{1}{C(n+m)}2^{m}\text{ }b^{k}b^{\varepsilon m}\geq \frac{1}{C(n+m)}%
M2^{-N}b^{\varepsilon m}.
\end{equation*}%
The assumption $\phi (n)$ $>>\log n$ guarantees that for large enough $n$,
depending on $\varepsilon ,$ $b^{\varepsilon m}>>n+m$. Consequently,
Corollary \ref{prob} implies%
\begin{equation*}
p_{i,k}\leq \exp \left( -c2^{m+k}\right)
\end{equation*}%
for a suitable constant $c>0$. Hence

\begin{equation*}
\sum_{i=1}^{2^{n}}\sum_{k=0}^{L}p_{i,k}\lesssim 2^{n}\exp \left(
-c2^{m}\right) .
\end{equation*}%
If we write $m=\phi (n)+J$ for $J\in \mathbb{N}$, then the term $2^{n}\exp
\left( -c2^{m}\right) $ is dominated by $\gamma ^{n+J}$ for some $\gamma <1$
because $\phi (n)>>\log n$. Hence for large enough $n$, the probability that
there is any interval $I_{n}^{(i)}$ at level $n$ and integer $m\geq \phi (n)$
with $N_{s_{n+m}}(I_{n}^{(i)}\tbigcap E)>2\left( \frac{s_{n}}{s_{n+m}}%
\right) ^{d}$ is at most 
\begin{eqnarray*}
\sum_{m=\phi (n)}^{\infty }\mathcal{P}\left( \exists \text{ }i\text{ with }%
N_{s_{n+m}}(I_{n}^{(i)}\tbigcap E)>2\left( \frac{s_{n}}{s_{n+m}}\right) ^{d}%
\text{ }\right) &\lesssim &\sum_{m=\phi (n)}^{\infty
}\sum_{i=1}^{2^{n}}\sum_{k=0}^{L-1}p_{i,k} \\
&\lesssim &\sum_{J=0}^{\infty }\gamma ^{n+J}\lesssim \gamma ^{n}\text{.}
\end{eqnarray*}

This shows that if $F_{n}$ is the event that there is any interval $I$ at
level $n$ and any $m\geq \phi (n)$ with $N_{s_{n+m}}(I\tbigcap E)>2\left(
s_{n}/s_{n+m}\right) ^{d},$ then

\begin{equation*}
\sum_{n=1}^{\infty }\mathcal{P}(F_{n})\lesssim \sum_{n=1}^{\infty }\gamma
^{n}<\infty .
\end{equation*}%
An application of the Borel Cantelli lemma proves that $\mathcal{P}\left(
F_{n}\text{ i.o.}\right) =0$ and that is what we desired to prove. Thus $%
\overline{\dim }_{\Phi }E\leq d$ almost surely.
\end{proof}

\subsection{Almost sure upper dimensions for \textquotedblleft
small\textquotedblright\ $\Phi $}

\label{subsec:uppersmall}

\begin{theorem}
\label{TheoremB} \label{asuppersmall} Let $a$ be a level comparable
sequence. For a.e. $E\in \mathcal{C}_{a}$ we have $\overline{\dim }_{\Phi
}E=1$ if 
\begin{equation*}
\Phi (x)<<\frac{\log \left\vert \log (x)\right\vert }{|\log x|}\text{ for }x%
\text{ near }0,
\end{equation*}%
equivalently, $\phi (n)<<\log n$.
\end{theorem}

Notice that if we take $\phi =0$ we get the result for the Assouad dimension.

\begin{corollary}
For almost all rearrangements $E\in \mathcal{C}_{a},\dim _{A}E=1$.
\end{corollary}

\begin{corollary}
The set of uniformly disconnected rearrangements in $\mathcal{C}_{a}$ is of
measure zero.
\end{corollary}

\begin{proof}
This is immediate from the fact that a subset $E$ of $\mathbb{R}$ is
uniformly disconnected (or porous) if and only if $\dim _{A}E<1$ (cf. \cite%
{MT}).
\end{proof}

\begin{proof}[Proof of Theorem \ref{TheoremB}] This will require us to prove that almost
surely there are $x_{n}\in E$, $R_{n}\rightarrow 0$ and $r_{n}\leq
R_{n}^{1+\Phi (R_{n})}$ satisfying%
\begin{equation*}
N_{r_{n}}(B(x_{n},R_{n})\tbigcap E)\gtrsim \left( \frac{R_{n}}{r_{n}}\right)
^{1-\varepsilon }
\end{equation*}%
for each fixed $\varepsilon >0$.

Put $\phi (n)=g(n)\log n$ where $g(n)\rightarrow 0$. Let $M_{n}(E)$ be the
random variable that counts the maximum number of gaps of levels $%
n+1,...,n+\phi (n)$ placed in any one of the intervals arising at step $n$
in the construction of $E$. Let 
\begin{equation*}
K_{n}=\frac{2\log 2^{n}}{\log \left( \frac{2^{n}\log 2^{n}}{2^{n+\phi (n)}}%
\right) }\sim \frac{n}{\log n}.
\end{equation*}%
There are a total of $\sim 2^{n+\phi (n)}$ of these gaps to place in the $%
2^{n}$ intervals. As 
\begin{equation*}
2^{n+\phi (n)}=2^{n}2^{g(n)\log n}=2^{n}n^{g(n)\log 2}<<2^{n}\log (2^{n}),
\end{equation*}%
Theorem 1 of \cite{RS} guarantees that $\mathcal{P}(M_{n}>K_{n})\geq 1/2$
for large enough $n$. (Our $K_{n}$ can be taken as $k_{1/2}$ in their
notation.)

Let $I_{n}$ denote the interval of step $n$ containing the maximum number of
gaps of levels $n+1,...,n+\phi (n)$ and let $L_{n}$ denote the sum of the
lengths of the gaps of levels deeper than $n+\phi (n)$ that are contained in 
$I_{n}$. Since the sequence $a$ is decreasing, 
\begin{equation*}
L_{n}+M_{n}a_{2^{n+\phi (n)}}\leq \text{length }I_{n}\leq
L_{n}+M_{n}a_{2^{n}}.
\end{equation*}

Put 
\begin{equation*}
R_{n}=L_{n}+M_{n}a_{2^{n}}\text{ and }r_{n}=a_{2^{n+\phi (n)+1}}.
\end{equation*}%
Then the ball centred at an endpoint $x_{n}$ of $I_{n}$ and radius $R_{n}$
contains $I_{n}$. As noted in the proof of the previous theorem, there is no
loss of generality in assuming the gaps are placed in $I_{n}$ in decreasing
order. Hence%
\begin{equation*}
N_{r_{n}}(B(x_{n},R_{n})\tbigcap E)\gtrsim M_{n}+\frac{L_{n}}{r_{n}}\text{,}
\end{equation*}%
while 
\begin{equation*}
\frac{R_{n}}{r_{n}}=M_{n}\frac{a_{2^{n}}}{a_{2^{n+\phi (n)}}}+\frac{L_{n}}{%
r_{n}}\lesssim M_{n}\tau ^{-\phi (n)}+\frac{L_{n}}{r_{n}},
\end{equation*}%
since $a_{2^{j}}/a_{2^{j+1}}\sim s_{j}/s_{j+1}\leq 1/\tau $ by the level
comparable assumption.

Fix $\varepsilon >0$ and consider%
\begin{eqnarray*}
\frac{N_{r_{n}}(B(x_{n},R_{n})\cap E)}{\left( \frac{R_{n}}{r_{n}}\right)
^{1-\varepsilon }} &\gtrsim &\frac{M_{n}+\frac{L_{n}}{r_{n}}}{\left(
M_{n}\tau ^{-\phi (n)}+\frac{L_{n}}{r_{n}}\right) ^{1-\varepsilon }}=\frac{%
M_{n}+\frac{L_{n}}{r_{n}}}{\left( M_{n}n^{g(n)|\log \tau |}+\frac{L_{n}}{%
r_{n}}\right) ^{1-\varepsilon }} \\
&\geq &\frac{M_{n}+\frac{L_{n}}{r_{n}}}{\left( M_{n}n^{g(n)|\log \tau
|}\right) ^{1-\varepsilon }+\left( \frac{L_{n}}{r_{n}}\right)
^{1-\varepsilon }}.
\end{eqnarray*}%
If $M_{n}\geq L_{n}/r_{n}$, then for large $n$ this ratio is at least%
\begin{equation*}
\frac{M_{n}^{\varepsilon }}{2n^{g(n)|\log \tau |(1-\varepsilon )}}\gtrsim 
\frac{M_{n}^{\varepsilon }}{n^{\varepsilon /2}}
\end{equation*}%
since $g(n)\rightarrow 0$. It follows that if $M_{n}\geq K_{n}$ and $n$ is
sufficiently large, then 
\begin{equation*}
\frac{N_{r_{n}}(B(x_{n},R_{n})\tbigcap E)}{\left( \frac{R_{n}}{r_{n}}\right)
^{1-\varepsilon }}\succsim \frac{n^{\varepsilon /2}}{\left( \log n\right)
^{\varepsilon }}\geq 1.
\end{equation*}%
Similar arguments give the same conclusion if, instead, $K_{n}\leq M_{n}\leq
L_{n}/r_{n}$.

We conclude that there are $x_{n},R_{n},r_{n}$ as outlined above, with%
\begin{equation*}
N_{r_{n}}(B(x_{n},R_{n})\tbigcap E)\gtrsim \left( \frac{R_{n}}{r_{n}}\right)
^{1-\varepsilon }
\end{equation*}%
whenever $M_{n}\geq K_{n}$. Since $M_{n}$ depends only on levels $%
n+1,...,n+\phi (n)$ and $K_{n}$ only on $n$, if we choose a sequence $%
n_{k}\rightarrow \infty $ such that $n_{k+1}>>n_{k}+\phi (n_{k})$, then the
sets $\{M_{n_{k}}\geq K_{n_{k}}\}$ are independent events, each occurring
with probability at least $1/2$. The Borel Cantelli lemma implies that these
events occur infinitely often with probability one. Thus almost surely 
\begin{equation*}
N_{r_{n}}(B(x_{n},R_{n})\tbigcap E)\gtrsim \left( \frac{R_{n}}{r_{n}}\right)
^{1-\varepsilon }
\end{equation*}%
for infinitely many $n$. Moreover, for such $n$ we have%
\begin{equation*}
R_{n}\gtrsim M_{n}a_{2^{n}}\gtrsim K_{n}s_{n}\gtrsim \frac{n}{\log n}s_{n}.
\end{equation*}%
Thus $(R_{n})^{1+\Phi (R_{n})}\gtrsim s_{n}^{1+\Phi (s_{n})}\sim s_{n+\phi
(n)}\sim r_{n}$.$\ $\texttt{\ }Also, since $M_{n}\lesssim 2^{n+\phi (n)},$ 
\begin{eqnarray*}
R_{n} &\lesssim &\sum_{i>n+\phi (n)}2^{i}a_{2^{i}}+M_{n}a_{2^{n}}\lesssim
\sum_{i>n+\phi (n)}(2\lambda )^{i}+2^{n+\phi (n)}\lambda ^{n} \\
&\lesssim &\left( 2\lambda \right) ^{n}2^{\phi (n)}=\left( 2\lambda \right)
^{n}n^{g(n)\log 2}\rightarrow 0\text{ as }n\rightarrow \infty
\end{eqnarray*}%
since $\lambda <1/2.$

This suffices to prove that with probability one, $\overline{\dim }_{\Phi
}E\geq 1-\varepsilon $ for each $\varepsilon >0$ and that completes the
argument.
\end{proof}

\section{Almost Sure Lower Dimensions for Complementary sets}

\label{sec:aslowerdim}

The almost sure results for lower $\Phi $-dimensions will again make use of
the probabilistic result, Corollary \ref{prob}, but will also use the fact
that it is \textquotedblleft quite likely\textquotedblright\ that some
interval at step $n$ will contain no gaps from levels $n+1,...,n+\log n$.
(This will be made precise in the proof.) In addition, for the case of
\textquotedblleft large\textquotedblright\ $\Phi ,$ we will also use
estimates on the size of the intervals created at step $n$ in the rearranged
set.

\subsection{Almost sure lower dimensions for \textquotedblleft
small\textquotedblright\ $\Phi $}

\label{subsec:lowersmall}

We will begin with the \textquotedblleft small\textquotedblright\ $\Phi $
case. We remark that any $E\in \mathcal{C}_{a}$ which admits an isolated
point has lower $\Phi $-dimension zero. However, these form a null set in $%
\mathcal{C}_{a}$ and thus are not of interest to us.

\begin{theorem}
\label{aslowersmall} Let $a=\{a_{n}\}$ be a level comparable sequence. For
almost every $E\in \mathcal{C}_{a}$ we have \underline{$\dim $}$_{\Phi }E=0$
if 
\begin{equation*}
\Phi (x)<<\frac{\log \left\vert \log (x)\right\vert }{|\log x|}\text{ for }x%
\text{ near }0,
\end{equation*}%
equivalently, $\phi (n)<<\log n$.
\end{theorem}

\begin{corollary}
For almost all rearrangements $E\in \mathcal{C}_{a},$ dim$_{L}E=0$ a.s.
\end{corollary}

\begin{corollary}
The set of uniformly perfect rearrangements in $\mathcal{C}_{a}$ is of
measure zero.
\end{corollary}

\begin{proof}
This follows as a set $E$ is uniformly perfect if and only if dim$_{L}E>0$,
by Lemma 2.1 in \cite{KLV}.
\end{proof}

\begin{proof}[Proof of Theorem \ref{aslowersmall}]
We will prove that for each $\varepsilon >0,$ \underline{$\dim $}$_{\Phi
}E\leq \varepsilon $ a.s. By definition, this is true if almost surely there
are $x_{n}\in E$, $R_{n}\rightarrow 0$ and $r_{n}\leq R_{n}^{1+\Phi (R_{n})}$
satisfying%
\begin{equation*}
N_{r_{n}}(B(x_{n},R_{n})\tbigcap E)\leq \left( \frac{R_{n}}{r_{n}}\right)
^{\varepsilon }\text{.}
\end{equation*}%
Choose $J\geq 1$ such that $s_{n}\leq a_{2^{n-J}}$ for all $n$. Put $%
R_{n}=s_{n}$ and $r_{n}=s_{n+\phi (n)}\leq R_{n}^{1+\Phi (n)}$. \ Label the
intervals arising at level $n-J$ in the construction of $E$ as $I_{n}^{(j)}$%
, $j=1,...,2^{n-J}$ and let $x_{n}=x_{n}(j)$ be an endpoint of interval $%
I_{n}^{(j)}$. The gap sizes ensure that $E\cap B(x_{n},R_{n})$ is contained
in $I_{n}^{(j)}$. Choose $D\in \mathbb{N}$ such that 
\begin{equation*}
\sum_{j=2^{n+\phi (n)+Dn}}^{\infty }a_{j}=2^{n+\phi (n)+Dn}s_{n+\phi
(n)+Dn}\leq r_{n}.
\end{equation*}

If the interval $I_{n}^{(j)}$ admits no gaps from levels $n-J+1,...,n+\phi
(n)+A\log n$ (where $A$ will be specified later) and $\Lambda _{k}^{(j)}$
gaps at each of levels $n+\phi (n)+k$ for $k=1+A\log n,....,Dn$, then 
\begin{equation*}
N_{r_{n}}(B(x_{n},R_{n})\tbigcap E)\leq N_{r_{n}}(I_{n}^{(j)}\tbigcap
E)\lesssim \sum_{k=1+A\log n}^{Dn}\Lambda _{k}^{(j)}\frac{s_{n+\phi (n)+k}}{%
s_{n+\phi (n)}}+1
\end{equation*}%
since the totality of the gaps of levels deeper than $n+\phi (n)+Dn$ can be
covered by one interval of radius $r_{n}$. We want to prove the quantity
above is bounded by $C\left( s_{n}/s_{n+\phi (n)}\right) ^{\varepsilon }$.

Let $F_{n}^{(j)}$ be the event that interval $I_{n}^{(j)}$ contains no gaps
from levels $n-J+1,...,n+\phi (n)+A\log n,$ but each interval $I_{n}^{(i)}$
for $i<j$ does contain at least one such gap. Let $G_{n}^{(j)}$ be the event
that 
\begin{equation}
\sum_{k=1+A\log n}^{Dn}\Lambda _{k}^{(j)}\frac{s_{n+\phi (n)+k}}{s_{n+\phi
(n)}}\leq 2\left( \frac{s_{n}}{s_{n+\phi (n)}}\right) ^{\varepsilon }.
\label{G}
\end{equation}%
If $F_{n}^{(j)}\tbigcap G_{n}^{(j)}$ is non-empty for some $j,$ then there
is a \textquotedblleft suitable\textquotedblright\ interval $I_{n}^{(j)},$
meaning, an interval at level $n-J$ which both admits no gaps from levels $%
n-J+1,...,n+\phi (n)+A\log n$ and has property (\ref{G}). As the events $%
F_{n}^{(j)}\tbigcap G_{n}^{(j)}$ are disjoint and the pairs $%
F_{n}^{(j)},G_{n}^{(j)}$ are independent (since the location of gaps at
different levels are independent), 
\begin{equation*}
\mathcal{P}(\exists \text{ suitable }I_{n}^{(j)})\geq \sum_{j=1}^{2^{n-J}}%
\mathcal{P}\left( F_{n}^{(j)}\tbigcap G_{n}^{(j)}\right) =\sum_{j}\mathcal{P}%
\left( F_{n}^{(j)}\right) \mathcal{P}\left( G_{n}^{(j)}\right) .
\end{equation*}

We first focus on $G_{n}^{(j)}$. For an appropriate constant $B>A+1,$ to be
specified later,%
\begin{eqnarray*}
\mathcal{P}\left( \left( G_{n}^{(j)}\right) ^{c}\text{ }\right) &\leq
&\sum_{k>A\log n}^{B\log n}\mathcal{P}\left( \Lambda _{k}^{(j)}\geq \left( 
\frac{s_{n}}{s_{n+\phi (n)}}\right) ^{\varepsilon }\frac{s_{n+\phi (n)}}{%
s_{n+\phi (n)+k}}\frac{1}{(B-A)\log n}\right) \\
&&+\sum_{k>B\log n}^{Dn}\mathcal{P}\left( \Lambda _{k}^{(j)}\geq \left( 
\frac{s_{n}}{s_{n+\phi (n)}}\right) ^{\varepsilon }\frac{s_{n+\phi (n)}}{%
s_{n+\phi (n)+k}}\frac{1}{Dn}\right) .
\end{eqnarray*}
If $k>A\log n$ then, since $s_{i}/s_{i+1}\geq 1/\tau >2$ for all $i,$ taking 
$\gamma =(2\tau )^{-1}>1$ gives 
\begin{eqnarray*}
\frac{1}{2^{k+\phi (n)}}\left( \frac{s_{n}}{s_{n+\phi (n)}}\right)
^{\varepsilon }\frac{s_{n+\phi (n)}}{s_{n+\phi (n)+k}}\frac{1}{\log n} &\geq
&\frac{\gamma ^{k}(2\tau ^{\varepsilon })^{-g(n)\log n}}{\log n} \\
&\geq &\frac{\gamma ^{k}n^{-g(n)\log (2\tau ^{\varepsilon })}}{\log n} \\
&\geq &\frac{n^{A\log \gamma -g(n)\log (2\tau ^{\varepsilon })}}{\log n}%
\rightarrow \infty
\end{eqnarray*}%
as $n\rightarrow \infty $. Thus if $n$ is sufficiently large, then 
\begin{equation*}
\left( \frac{s_{n}}{s_{n+\phi (n)}}\right) ^{\varepsilon }\frac{s_{n+\phi
(n)}}{s_{n+\phi (n)+k}}\frac{1}{(B-A)\log n}>>2^{k+\phi (n)}=\mathbb{E(}%
\Lambda _{k}^{(j)}).
\end{equation*}

Similarly, if $k>B\log n$, then 
\begin{eqnarray*}
\frac{1}{2^{k+\phi (n)}}\left( \frac{s_{n}}{s_{n+\phi (n)}}\right)
^{\varepsilon }\frac{s_{n+\phi (n)}}{s_{n+\phi (n)+k}}\frac{1}{Dn} &\geq &%
\frac{\gamma ^{k}(2\tau ^{\varepsilon })^{-g(n)\log n}}{Dn} \\
&\geq &\frac{n^{B\log \gamma -g(n)\log (2\tau ^{\varepsilon })}}{Dn}%
\rightarrow \infty ,
\end{eqnarray*}%
provided we choose $B$ so large that $B\log \gamma >1$. Hence, again, we
conclude that

\begin{equation*}
\left( \frac{s_{n}}{s_{n+\phi (n)}}\right) ^{\varepsilon }\frac{s_{n+\phi
(n)}}{s_{n+\phi (n)+k}}\frac{1}{Dn}>>2^{k+\phi (n)}=\mathbb{E(}\Lambda
_{k}^{(j)}).
\end{equation*}

Appealing to Corollary \ref{prob}, we deduce that%
\begin{eqnarray*}
\mathcal{P}\left( \left( G_{n}^{(j)}\right) ^{c}\right) &\leq &\sum_{k>A\log
n}^{B\log n}\exp (-c2^{k+\phi (n)})+\sum_{k>B\log n}^{Dn}\exp (-c2^{k+\phi
(n)}) \\
&\lesssim &\exp (-c2^{A\log n+\phi (n)})\leq 1/2
\end{eqnarray*}%
for large enough $n$. Thus $\mathcal{P(}G_{n}^{(j)})\geq 1/2$ for each $j.$

Next, we observe that $\sum_{j}\mathcal{P(}F_{n}^{(j)})$ is the probability
of there being an interval at level $n-J$ with no gaps from levels $%
n-J+1,...,n+\phi (n)+A\log n$. This is mathematically the same as the
problem of equally distributing $2^{n+\phi (n)+A\log n}-2^{n-J}\sim
2^{n+\phi (n)+A\log n}$ balls into $2^{n-J}$ bins and asking if one of the
bins is empty. The expected number of balls in a bin is $\sim 2^{\phi
(n)+A\log n}\leq n^{p}$ for $p<1$ if we choose $A$ sufficiently small and $n$
large. Thus 
\begin{equation*}
2^{n}\exp (-\text{Expected \# balls)}\geq 2^{n}\exp (-Cn^{p})\rightarrow
\infty .
\end{equation*}%
According to \cite[p. 111, Theorem 4]{Kbook}, the probability that there is
an empty bin tends to $1$ as $n\rightarrow \infty $.

Thus for $n$ sufficiently large, $\sum_{j}\mathcal{P(}F_{n}^{(j)})$ $\geq
1/2 $ and hence 
\begin{equation*}
\mathcal{P}(\exists \text{ suitable }I_{n}^{(j)})\geq \sum_{j}\mathcal{P}%
\left( F_{n}^{(j)}\right) \mathcal{P}\left( G_{n}^{(j)}\right) \geq 1/4
\end{equation*}%
if $n$ is large and $A,B$ are chosen suitably. Furthermore, this probability
depends only upon the placement of the gaps at levels $n-J+1,..,Dn,$ hence
if we pick a subsequence $\{n_{k}\}\rightarrow \infty $ with $%
n_{k+1}-J>>Dn_{k}$, these events are independent. By the Borel Cantelli
lemma the events occur infinitely often with probability one. In other
words, with probability one there are choices $x_{n}\in E$, $%
R_{n}\rightarrow 0$ and $r_{n}\leq R_{n}^{1+\Phi (R_{n})}$ satisfying%
\begin{equation*}
N_{r_{n}}(B(x_{n},R_{n})\tbigcap E)\leq \left( \frac{R_{n}}{r_{n}}\right)
^{\varepsilon }\text{ i.o.}
\end{equation*}%
and, as we observed at the beginning of the proof, this is sufficient to
show \underline{$\dim $}$_{\Phi }E=0$ a.s.
\end{proof}

\subsection{Almost sure lower dimensions for \textquotedblleft
large\textquotedblright\ $\Phi $}

\label{subsec:lowerlarge}

Before turning to the \textquotedblleft large\textquotedblright\ $\Phi $
case, we first establish a bound on the almost sure length of the intervals
of level $n$. This lemma will be useful in determining the almost sure
behaviour of the lower $\Phi $-dimensions because it will allow us to use
lower bounds for the covering numbers of intervals from the construction, in
place of covering numbers of arbitrary balls.

\begin{lemma}
For a.e. $E\in \mathcal{C}_{a},$ the maximum length of any interval of level 
$n$ in the construction of $E$ is at most $3Cs_{n}^{1-\varepsilon _{n}}$
where $C$ is chosen to satisfy $Cs_{j}\geq a_{2^{j-1}}$ for all $j$, $%
\varepsilon _{n}=(4\log n)/n$ and $n$ is sufficiently large.
\end{lemma}

\begin{proof}
Choose $D$ such that $(2\lambda )^{D}\leq \tau ,$ so that $2^{Dn}s_{Dn}\leq
s_{n}$. Let $I_{n}^{(j)}$ be an interval of level $n$ in $E$. Denote 
\begin{eqnarray*}
\Lambda _{0}^{(j)} &=&\#\text{ gaps in }I_{n}^{(j)}\text{ from levels }%
n+1,...,n(1+\varepsilon _{n}/2)\text{ and} \\
\Lambda _{k}^{(j)} &=&\#\text{ gaps in }I_{n}^{(j)}\text{ from levels }%
k=1+n(1+\varepsilon _{n}/2),..,Dn.
\end{eqnarray*}%
Since any gap of level $i$ has length at most $a_{2^{i-1}}\leq Cs_{i}$, the
length of an interval $I_{n}^{(j)}$ of level $n$ is bounded by%
\begin{equation*}
L_{j}=\text{length}(I_{n}^{(j)})\leq C\left( \Lambda
_{0}^{(j)}s_{n}+\sum_{k=1}^{Dn-n(1+\varepsilon _{n}/2)}\Lambda
_{k}^{(j)}s_{n(1+\varepsilon _{n}/2)+k}+2^{Dn}s_{Dn}\right) .
\end{equation*}%
Hence, if $L_{j}>3Cs_{n}^{1-\varepsilon _{n}},$ then since $%
\sum_{k>Dn}2^{k}s_{k}\leq s_{n}^{1-\varepsilon _{n}},$ either $\Lambda
_{0}^{(j)}s_{n}>s_{n}^{1-\varepsilon _{n}},$ or for some $%
k=1,...,Dn-n(1+\varepsilon _{n}/2)$ $,$%
\begin{equation*}
\Lambda _{k}^{(j)}s_{n(1+\varepsilon _{n}/2)+k}\geq \frac{%
s_{n}^{1-\varepsilon _{n}}}{Dn}.
\end{equation*}%
In other words, either 
\begin{equation*}
\Lambda _{0}^{(j)}\geq s_{n}^{-\varepsilon _{n}}\text{, }
\end{equation*}%
or for some $k\leq Dn-n(1+\varepsilon _{n}/2)$, 
\begin{equation*}
\Lambda _{k}^{(j)}\geq \frac{s_{n}^{1-\varepsilon _{n}}}{Dns_{n(1+%
\varepsilon _{n}/2)+k}}.
\end{equation*}%
In comparison, the expected value of $\Lambda _{0}^{(j)}\sim 2^{n\varepsilon
_{n}/2}=n^{\log 4}<<s_{n}^{-\varepsilon _{n}}$ and the expected value of $%
\Lambda _{k}^{(j)}\sim 2^{k+n\varepsilon _{n}/2}=2^{k}n^{\log 4},$ while 
\begin{eqnarray*}
\frac{s_{n}^{1-\varepsilon _{n}}}{Dns_{n(1+\varepsilon _{n}/2)+k}} &\geq &%
\frac{s_{n}2^{n\varepsilon _{n}}}{Dns_{n(1+\varepsilon _{n}/2)+k}}\geq \frac{%
2^{3n\varepsilon _{n}/2+k}}{Dn}=\frac{2^{6\log n+k}}{Dn} \\
&\sim &2^{k}n^{6\log 2-1}>>2^{k}n^{\log 4}.
\end{eqnarray*}%
Appealing to Corollary \ref{prob}, we deduce that%
\begin{equation*}
\mathcal{P}\left( \Lambda _{0}^{(j)}s_{n}\geq s_{n}^{1-\varepsilon
_{n}}\right) \leq \exp (-cn^{\log 4})
\end{equation*}%
and%
\begin{equation*}
\mathcal{P}\left( \Lambda _{k}^{(j)}s_{n(1+\varepsilon _{n}/2)+k}\geq
s_{n}^{1-\varepsilon _{n}}/Dn\right) \leq \exp (-c2^{k}n^{\log 4}).
\end{equation*}%
Thus the probability that $L_{j}$ is more than $3Cs_{n}^{1-\varepsilon _{n}}$
is at most%
\begin{equation*}
\sum_{k=0}^{Dn}\exp (-c2^{k}n^{\log 4})\lesssim \exp (-cn^{\log 4}).
\end{equation*}%
Therefore the probability that any of the $2^{n}$ intervals at level $n$ has
length exceeding $3Cs_{n}^{1-\varepsilon _{n}}$ is at most $2^{n}\exp
(-cn^{\log 4}),$ and that decays exponentially in $n$. Applying the Borel
Cantelli lemma it follows that for almost all $E\in \mathcal{C}_{a}$, all
intervals of level $n$ have length at most $3Cs_{n}^{1-\varepsilon _{n}},$
for large enough $n$.
\end{proof}

\begin{theorem}
\label{aslowerlarge} Let $a$ be a level comparable sequence. For a.e. $E\in 
\mathcal{C}_{a}$ we have \underline{$\dim $}$_{\Phi }E=$ \underline{$\dim $}$%
_{\Phi }C_{a}$ if 
\begin{equation*}
\Phi (x)>>\frac{\log \left\vert \log (x)\right\vert }{|\log x|}\text{ for }x%
\text{ near }0,
\end{equation*}%
equivalently, $\phi (n)>>\log n$.
\end{theorem}

By the same reasoning as Corollary \ref{qA} we have

\begin{corollary}
For almost all rearrangements $E\in \mathcal{C}_{a},$ dim$_{qL}E=\dim
_{qL}C_{a}$.
\end{corollary}

\begin{proof}[Proof of Theorem \ref{aslowerlarge}]
Choose $d<\underline{\dim }_{\Phi }C_{a}$. We will show that \underline{$%
\dim $}$_{\Phi }E\geq d$ a.s. Since it was already seen in Theorem 4.3 of 
\cite{GHMPhi} (see Theorem \ref{Recall}) that always \underline{$\dim $}$%
_{\Phi }E\leq $ \underline{$\dim $}$_{\Phi }C_{a},$ this will complete the
proof.

From the previous lemma, we know that for all $E\in \Omega ^{\prime },$ a
subset of the probability space $\Omega $ with full measure, all intervals
at level $n$ (in the construction of $E)$ have length at most $%
3Cs_{n}^{1-\varepsilon _{n}}$ for $\varepsilon _{n}=(4\log n)/n$ and for $n$
sufficiently large. Our task is to prove that for almost every $E\in \Omega
^{\prime }$, we have $N_{r}(B(x,R)\tbigcap E)\geq \left( \frac{R}{r}\right)
^{d}$ for all $x\in E,$ small $R$ and $r\leq R^{1+\Phi (R)}$.

It suffices to consider $R=R_{n}=3Cs_{n}^{1-\varepsilon _{n}}$ (where $C$ is
as in the previous lemma) since the definition of $\varepsilon _{n}$ implies
that there are positive constants $a,b$ such that $a\leq
s_{n}^{1-\varepsilon _{n}}/s_{n+1}^{1-\varepsilon _{n+1}}\leq b$. Choose $m$
such that $s_{m+1}\leq R\leq s_{m}$. Notice that as $s_{n}/s_{m}\in \lbrack
\tau ^{n-m},\lambda ^{n-m}]$ and $s_{n}^{\varepsilon _{n}}\in \lbrack \tau
^{4\log n},\lambda ^{4\log n}]$, $n-m\sim \log n$.

We may assume $r=r_{n}=s_{m+\phi (m)+k}$ for some $k\geq 0$.

If $x$ belongs to the level $n$ interval $I_{n}^{(j)}$, then $%
B(x,R)\supseteq I_{n}^{(j)}$. Hence it will be enough to show that 
\begin{equation*}
N_{r}(I_{n}^{(j)}\tbigcap E)\gtrsim \left( \frac{s_{m}}{s_{m+\phi (m)+k}}%
\right) ^{d}\sim \left( \frac{R}{r}\right) ^{d}\text{ }
\end{equation*}%
for all large $n$.

From the formula for the lower $\Phi $-dimension of $C_{a}$ (\ref{LowerTheta}%
), we know that for $\delta >0$ chosen such that $d+\delta <\underline{\dim }%
_{\Phi }C_{a}$ and large enough $m$, 
\begin{equation*}
\left( \frac{s_{m}}{s_{m+\phi (m)+k}}\right) ^{d+\delta }\leq 2^{\phi (m)+k},
\end{equation*}%
thus it will be enough to check that 
\begin{equation*}
N_{r_{n}}(I_{n}^{(j)}\tbigcap E)\gtrsim 2^{(\phi (m)+k)d/(d+\delta )}.
\end{equation*}

Choose $J$ such that $a_{2^{i-J}}\geq 2s_{i}$ for all $i$. Then $%
N_{r_{n}}(I_{n}^{(j)}\tbigcap E)$ will be at least the number of gaps of
level $m+\phi (m)+k-J$ contained in $I_{n}^{(j)}$ as such gaps have length
at least $2r_{n}$. The expected number of gaps of level $m+\phi (m)+k-J$
contained in $I_{n}^{(j)}$ is at least%
\begin{equation*}
\frac{2^{m+\phi (m)+k-J}}{2^{n}}\sim 2^{\phi (m)+k-b_{n}\log n}
\end{equation*}%
for some $b_{n}$ bounded above and below from $0$. Since $\phi (m)=f(m)\log
m $ where $f(m)\rightarrow \infty $, 
\begin{equation*}
\frac{2^{\phi (m)+k-b_{n}\log n}}{2^{(\phi (m)+k)d/(d+\delta )}}\geq
2^{-b_{n}\log n}2^{\delta (f(m)\log m)/(d+\delta )}\rightarrow \infty
\end{equation*}%
as $m$ $\rightarrow \infty $ (or equivalently, $n\rightarrow \infty $ since $%
m\sim n$). Corollary \ref{prob} implies that the probability that $%
N_{r}(I_{n}^{(j)}\tbigcap E)<2^{(\phi (m)+k)d/(d+\delta )}$ for some $k\geq
1 $ and $j=1,...,2^{n}$ is at most 
\begin{eqnarray*}
\sum_{j=1}^{2^{n}}\sum_{k=1}^{\infty }\exp (-c2^{\phi (m)+k-b_{n}\log n})
&\lesssim &2^{n}\exp (-c2^{\phi (m)-b_{n}\log n}) \\
&=&2^{n}\exp (-c2^{f(m)\log m-b_{n}\log n})\leq \gamma ^{n}
\end{eqnarray*}%
for some $\gamma <1$ since $m\sim n$ and $f(m)\rightarrow \infty $.

Applying the Borel Cantelli lemma again, the probability that there are some 
$x_{n}$ with $N_{r_{n}}(B(x_{n},R_{n})\tbigcap E)\lesssim \left( \frac{R_{n}%
}{r_{n}}\right) ^{d}$ i.o. is zero. That completes the proof.
\end{proof}

\begin{remark}
It would be interesting to know what happens if $\phi (n)/\log n$ does not
tend to either 0 or infinity. Even for the case $\Phi =\log \left\vert \log
x\right\vert /\left\vert \log x\right\vert $ we do not know if the $\Phi $
dimensions almost surely coincide with the dimension of the Cantor set, the
dimension of the decreasing set or something else altogether.
\end{remark}


\begin{thebibliography}{99}
\bibitem{A1} P. Assouad, U.E..R. Math\'{e}matique, Universit\'{e} Paris XI,
Orsay. Th{\`{e}}se de doctorat d'\'{E}tat, \emph{Publications Math{\'{e}}%
matiques d'Orsay,} No. 223-7769, 1977.

\bibitem{A2} P. Assouad. \'{E}tude d'une dimension m\'{e}trique li\'{e}e 
\`{a} la possibilit\'{e} de plongements dans {$\mathbf{R}\sp{n}$}, \emph{C.
R. Acad. Sci. Paris S\'{e}r. A-B,} \textbf{288}(15):A731--A734, 1979.

\bibitem{BT} A.S. Besicovitch and S.J. Taylor. On the complementary
intervals of a linear closed set of zero {L}ebesgue measure, \emph{J. London
Math. Soc.,} \textbf{29}:449--459, 1954.

\bibitem{Bo} B. Bollobas. Random graphs, \emph{Academic Press, London}, 1985.

\bibitem{CDW} H. Chen, Y. Du and C. Wei. Quasi-lower dimension and
quasi-Lipschitz mapping, \emph{Fractals,} \textbf{25}(3), 1-9, 2017.

\bibitem{Fal} K. Falconer. Techniques in fractal geometry, \emph{John Wiley
\& Sons Ltd., Chichester,} 1997.

\bibitem{FTrans} J.M. Fraser. Assouad type dimensions and homogeneity of
fractals, \emph{Trans. Amer. Math. Soc.,} \textbf{366}(12):6687--6733, 2014.

\bibitem{FTale} J. M. Fraser, K.G. Hare, K.E. Hare, S. Troscheit and H. Yu.
The Assouad spectrum and the quasi-Assouad dimension: a tale of two spectra, 
\emph{Ann. Acad. Sci. Fenn. Math., to appear}. \emph{arXiv preprint arXiv}:
1804.096, 2018.

\bibitem{FYAdv} J. M. Fraser and H. Yu. New dimension spectra: finer
information on scaling and homogeneity, \emph{Adv. Math., }\textbf{329}%
:273--328, 2018.

\bibitem{GH} I. Garc\'i{a} and K. Hare. Properties of Quasi-Assouad
dimension, \emph{Camb. Phil. Soc., to appear. arXiv preprint arXiv}:
1703.02526, 2017.

\bibitem{GHM} I. Garc\'{\i}a, K.E. Hare and F. Mendivil. Assouad dimensions
of complementary sets, \emph{Proc. Roy. Soc. Edinburgh Sect. A} \textbf{148}%
:57-540, 2018.

\bibitem{GHMPhi} I. Garc\'{\i}a, K.E. Hare and F. Mendivil. Intermediate
Assouad-like dimensions, {\em arXiv preprint arXiv 1903.07155}, 2019.

\bibitem{HMZ} K.E. Hare, F. Mendivil, and L. Zuberman. The sizes of
rearrangements of Cantor sets, \emph{Can. Math. Bull.,} \textbf{56}%
(2):354--365, 2013.

\bibitem{Haw} J. Hawkes. Random re-orderings of intervals complementary to a
linear set, \emph{Quart. J. Math. Oxford Ser.,} \textbf{35}:165-172, 1984.

\bibitem{Hu1} Hu, X. The exact Hausdorff measure for a random re-ordering of
the Cantor set, \emph{Sci. China Ser. A,} \textbf{38}(3)\textbf{:}273-286,
1995.

\bibitem{Hu2} Hu, X. The exact packing measure for a random re-ordering of
the Cantor set, \emph{Sci. China Ser. A, }\textbf{39}(1)\textbf{:}1-6, 1996.

\bibitem{KLV} A. K{\"{a}}enm{\"{a}}ki, J. Lehrb{\"{a}}ck and M. Vuorinen.
Dimensions, Whitney covers, and tubular neighborhoods, \emph{\ Indiana
University Mathematics J.,} \textbf{62}:1861--1889, 2013.

\bibitem{KR} A. K\"{a}enm\"{a}ki and E. Rossi, Weak separation condition,
Assouad dimension, and Furstenberg homogeneity. \emph{Ann. Acad. Sci. Fenn.
Math}., \textbf{41}:465-490, 2016.

\bibitem{Kbook} V. Kolchin, B. Sevast'yanov and V. Chistyakov. Random
Allocations. \emph{V. H. Winston \& Sons, New York}, 1978.

\bibitem{L1} D.G. Larman. \newblock A new theory of dimension. \newblock%
\emph{Proc. Lond. Math. Soc.,} \textbf{3}(1):178--192, 1967.

\bibitem{LX} F. L\"{u} and L. Xi. Quasi-Assouad dimension of fractals. \emph{%
J. Fractal Geom.,} \textbf{3}(2):187-215, 2016.

\bibitem{Luu} J. Luukkainen. Assouad dimension: antifractal metrization,
porous sets, and homogeneous measures, \emph{J. Korean Math Soc.,} \textbf{35%
}(1):23--76, 1998.

\bibitem{MT} J. Mackay and J. Tyson. Conformal dimension. \emph{Univ.
Lecture Series} \textbf{54}, 2010.

\bibitem{RS} M. Raab and A. Steger. \textquotedblleft Balls into Bins" - A
simple and tight analysis. Randomization and approximation techniques in
computer science (Barcelona 1998). \emph{Lecture Notes in Comp. Sci., }%
\textbf{1518.}\emph{\ Springer, Berlin, Heidelberg,} 159-170, 1998.
\end{thebibliography}
\end{document}